\documentclass[11pt]{article}

\usepackage[margin=2cm]{geometry}

\usepackage{amsmath}
\usepackage{amssymb}
\usepackage{amsthm}
\usepackage{amsfonts}
\usepackage{enumerate}
\usepackage{graphicx}
\usepackage{color}
\usepackage{verbatim}

\usepackage{multicol}
\usepackage{mathrsfs}
\usepackage{mathtools}

\usepackage{ytableau}

\usepackage{subfig}
\usepackage{tikz}

\usetikzlibrary{arrows,shapes,positioning,decorations.markings}

\tikzset{->-/.style={decoration={
  markings,
  mark=at position .5 with {\arrow{>}}},postaction={decorate}}}

\usepackage{array}

\usepackage{xcolor,colortbl}

\usepackage{soul}

\newtheorem{theorem}{Theorem}[section]

\newtheorem{lemma}[theorem]{Lemma}

\newtheorem{corollary}[theorem]{Corollary}

\newtheorem{proposition}[theorem]{Proposition}

\begin{document}

\title{Inversions of Semistandard Young Tableaux}
\author{Paul Drube\\
\small Department of Mathematics and Statistics\\
\small Valparaiso University\\
\small \tt paul.drube@valpo.edu 
}

\maketitle

\begin{abstract}
A tableau inversion is a pair of entries from the same column of a row-standard tableau that lack the relative ordering necessary to make the tableau column-standard.  An $i$-inverted Young tableau is a row-standard tableau with precisely $i$ inversion pairs, and may be interpreted as a generalization of (column-standard) Young tableau.  Inverted Young tableau that lack repeated entries were introduced by Fresse to calculate the Betti numbers of Springer fibers in Type A, and were later developed as combinatorial objects in their own right by Beagley and Drube.  This paper generalizes earlier notions of tableau inversions to row-standard tableaux with repeated entries, yielding an interesting new generalization of semistandard (as opposed to merely standard) Young tableaux.  We develop a closed formula for the maximum numbers of inversion pairs for a row-standard tableau with a specific shape and content, and show that the number of $i$-inverted tableaux of a given shape is invariant under permutation of content.  We then enumerate $i$-inverted Young tableaux for a variety of shapes and contents, and generalize an earlier result that places $1$-inverted Young tableaux of a general shape in bijection with $0$-inverted Young tableaux of a variety of related shapes.
\end{abstract}

\section{Introduction}
\label{sec: intro}

Consider the non-increasing sequence of positive integers $\lambda = (\lambda_1,\lambda_2,\hdots,\lambda_m)$, and let $N = \lambda_1 + \hdots + \lambda_m$.  A \textbf{Young diagram} $Y$ of shape $\lambda$ is a left-justified array of $N$ total boxes such that there are $\lambda_i$ boxes in the $i^{th}$ row of $Y$.  A \textbf{(semistandard) filling} of a Young diagram $Y$ is an assignment of positive integers (possibly repeated) to the boxes of $Y$ such that integers strictly increase from left-to-right across each row and weakly increase from top-to-bottom down each column.  We assume that no positive integers are skipped, so that the boxes of $Y$ are filled with $1,2,\hdots,M$ for some $M \leq N$.  We call the resulting array $T$ a \textbf{semistandard Young tableau} of shape $\lambda$.  If each of $1,2,\hdots,N$ appears precisely once in $T$, the semistandard Young tableau $T$ qualifies as a \textbf{standard Young tableau} of shape $\lambda$.  In this paper we will also need to consider a generalization of semistandard fillings where integers strictly increase from left-to-right across each row but no longer need to weakly increase down each column.  We refer to such an array as a \textbf{row-standard tableau}.

If $\mu = (\mu_1,\mu_2,\hdots,\mu_M)$ is an ordered partition of $N$, we say that a semistandard tableau $T$ of shape $\lambda$ has \textbf{content} $\mu$ if its boxes are filled with precisely $\mu_1$ copies of $1$, $\mu_2$ copies of $2$, etc.  We often use the abbreviated notation $\mu = 1^{\mu_1} 2^{\mu_2} \hdots M^{\mu_M}$.  Thus a standard Young tableau is simply a semistandard Young tableau with content $\mu = 1^1 2^1 \hdots N^1$.  We denote the entire set of semistandard Young tableaux with shape $\lambda$ and content $\mu$ by $S(\lambda,\mu)$, and the set of standard Young tableaux with shape $\lambda$ by $S(\lambda)$.  For a great introduction to Young tableaux, see \cite{Fulton}.

Now consider the permutation $\sigma \in S_n$.  An inversion of $\sigma$ is a pair of integers $ i,j$ satisfying $i<j$ and $\sigma(i) > \sigma(j)$.  In this situation we call $(i,j)$ an inversion pair of $\sigma$.  Denote the number of distinct inversion pairs of $\sigma$ by $n_{inv}(\sigma)$.

As introduced by Fresse in \cite{Fresse}, permutation inversions admit a generalization to row-standard tableaux with non-repeated entries.  Let $Y$ be a Young diagram of shape $\lambda$ whose boxes have been filled with $1,2,\hdots,N$ to produce the row-standard tableau $\tau$.  Following \cite{Fresse}, a pair of entries $i,j$ from the same column of $\tau$ participate in an \textbf{inversion} of $\tau$ if $i<j$ and either of the following conditions hold:
\begin{enumerate}
\item At least one of $i$ and $j$ lacks an entry directly to its right, and $i$ is below $j$.
\item $i$ is directly to the left of $i'$, $j$ is directly to the left of $j'$, and $i' > j'$.
\end{enumerate}

In this situation, we write $(i,j)_\tau$ or simply $(i,j)$ and say that $i,j$ constitute a single inversion pair of $\tau$.  If a row-standard tableau $\tau$ has precisely $K$ distinct inversion pairs we write $n_{inv}(\tau) = K$.  Notice that a row-standard tableau $\tau$ is also column-standard and hence is a standard Young tableau if and only if $n_{inv}(\tau) = 0$.  Also notice that our definition of tableau inversion specializes to the earlier notion of permutation inversion if one interprets $\sigma$ as a single-column row-standard tableau whose entries appear as $\sigma(1),\hdots,\sigma(n)$ from top-to-bottom. 

As shown in \cite{Fresse}, for any row-standard tableau $\tau$ without repeated entries one can always recursively eliminate inversions to produce a unique column-standard tableau with no inversions.  The resulting standard Young tableau is known as the \textbf{standardization} of $\tau$ and is written $st(\tau)$. As any such $\tau$ may transformed into a standard Young tableau $st(\tau)$ by recursively removing inversions, we henceforth refer to the row-standard $\tau$ as an \textbf{inverted (standard) Young tableau} based on $st(\tau)$.  In Figure \ref{fig: standard inverted example} we show an inverted tableau of shape $\lambda = (4,3,2)$ alongside its standardization.  For a given shape $\lambda$, we denote the set of all inverted standard Young tableaux of shape $\lambda$ with precisely $i$ inversions by $S_i(\lambda)$.  Thus $S(\lambda) = S_0(\lambda)$.  We more specifically refer to elements of $S_i(\lambda)$ as \textbf{i-inverted (standard) Young tableaux} of shape $\lambda$.

\begin{figure}[h!]
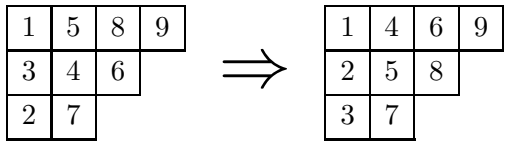

\centering
\begin{ytableau}
1 & 5 & 8 & 9 \\
3 & 4 & 6 \\
2 & 7 
\end{ytableau}
\hspace{.08in}
\raisebox{-14pt}{\scalebox{2.5}{$\Rightarrow$}}
\hspace{.08in}
\begin{ytableau}
1 & 4 & 6 & 9\\
2 & 5 & 8 \\
3 & 7
\end{ytableau}
\caption{An inverted tableau with inversion pairs $(6,8)$, $(1,3)$, $(2,3)$ and its standardization}
\label{fig: standard inverted example}
\end{figure}

Fresse introduced tableau inversions in \cite{Fresse} to calculate the Betti numbers of Springer fibers in type A.  Fixing the standard Young tableau $T$ of shape $\lambda$, he showed that the component of the Springer variety $F_\lambda$ associated with $T$ has $m^{th}$ Betti number equal to the number of $(d-m)$-inverted Young tableaux based on $T$, where $d$ is the dimension of the entire Springer variety.  Fresse also presents an algorithm for determining the number of $i$-inverted Young tableaux based on a specific standard tableau $T$.

In \cite{BD}, the author and Beagley present results enumerating the total number of $i$-inverted Young tableaux of shape $\lambda$, simultaneously ranging over all underlying standardizations.  By \cite{Fresse}, this yielded easily calculable formulas for the Betti numbers of the entire Springer variety $F_\lambda$ in a number of interesting cases.  In particular, \cite{BD} gives closed formulas for $\vert S_1(\lambda) \vert$, $\vert S_{M-1}(\lambda) \vert$, and $\vert S_{M-2}(\lambda) \vert$, where $M$ is the maximum number of inversions possible for any inverted Young tableau of shape $\lambda$.  That same paper also presents closed formulas for general $\vert S_i(\lambda) \vert$ in the case of relatively ``easy" choices for $\lambda$.

The combinatorial results of \cite{BD} also formalized earlier work on the Bar-Natan skein module of the solid torus presented by Russell in \cite{Russell}, with the generators of Russell's skein module standing in bijection with inverted Young tableaux of shape $\lambda=(n,n)$.  It is hypothesized that Russell's work extends to the $sl_n$ skein module of the solid torus for all $n \geq 2$, giving an interesting topological interpretation of inverted Young tableaux for any rectangular shape $\lambda = (n,\hdots,n)$.  An upcoming paper by the author \cite{Drube2} explicitly demonstrates this correspondence in the $n=3$ case, whereas another upcoming paper \cite{Drube1} adapts the m-diagram techniques of Tymoczko \cite{Tymoczko} to give a related geometric of inverted Young tableaux in terms of certain classes of marked planar graphs.

The primary goal of this paper is to generalize the notion of tableau inversions to the semistandard case, where repeated entries are possible, and to investigate which results from \cite{Fresse} and \cite{BD} extend to this more sophisticated case.  Although the algebraic geometry of this case has not been explicitly worked out, seeing as Spaltenstein varieties are the generalization of Springer varieties corresponding to semistandard Young tableaux, the author suspects that this paper may shed light on the Betti numbers of Spaltenstein varieties.  In the spirit of \cite{Russell} and \cite{Drube2}, the author also suspects that this semistandard generalization will be topologically realized by skein modules of the solid torus where the boundary circles are not consistently oriented.  Do note that the focus of this paper is purely combinatorial; no knowledge of algebraic varieties or skein modules is required, and Springer/Spaltenstein varieties will only be mentioned in passing.

So let $\tau$ be a row-standard tableau of shape $\lambda$ and content $\mu$, and let $i,j$ be a pair of entries from the same column of $\tau$.  Let $\lbrace i_1,i_2,\hdots \rbrace$ denote the (possibly empty) sequence of entries directly to the right of $i$ in $\tau$, read from left-to-right, and let $\lbrace j_1,j_2,\hdots \rbrace$ denote the (possibly empty) sequence of entries directly to the right of $j$, read from left-to-right.  We assert that $i,j$ participate in an \textbf{inversion} of $\tau$ if $i < j$ and one of the following holds:

\begin{enumerate}
\item At least one of $i_1$ and $j_1$ doesn't exist, and $i$ is below $j$.
\item $i_1$ and $j_1$ both exist, and $i_1 > j_1$.
\item $i_k$ and $j_k$ both exist for all $k \leq N$ with $i_k = j_k$ for all $k \leq N$, at least one of $i_{N+1}$ or $j_{N+1}$ doesn't exist, and $i$ is below $j$.
\item $i_k$ and $j_k$ both exist for all $k \leq N$ with $i_k = j_k$ for all $k \leq N$, $i_{N+1}$ and $j_{N+1}$ both exist, and $i_{N+1} > j_{N+1}$.
\end{enumerate}

In Section \ref{sec: basic results}, the somewhat redundant definition above will be streamlined utilizing what we refer to as the ``height order" on tableau entries.  The reason for the lengthier set of conditions above is that it betrays how our notion is a direct generalization of tableau inversions for standard tableaux: as $i_k = j_k$ is impossible in the case of non-repeated entries, only the first two conditions above are relevant in that situation.  If any of the conditions above hold, we once again write $(i,j)_\tau$ or simply $(i,j)$ and say that $i,j$ constitute an inversion pair of $\tau$.  We also retain our notation that $n_{inv}(\tau)$ denotes the total number of distinct inversion pairs in $\tau$.  In this case, the row-standard $\tau$ qualifies as a semistandard Young tableau if and only if $n_{inv}(\tau) = 0$.

A direct generalization of the technique from \cite{Fresse} shows that one may recursively remove inversions in any row-standard tableau $\tau$ to produce a column-standard semistandard tableau with no inversions, which we again refer to as the standardization $st(\tau)$ of $\tau$.  This standardization is merely the semistandard tableau where one has independently reordered the entries in each column so that they are weakly-increasing from top-to-bottom, and is guaranteed to be row-standard if the original $\tau$ was row-standard.  This fact prompts our definition of $\tau$ as an \textbf{inverted semistandard Young tableau} based on $st(\tau)$.  Figure \ref{fig: semistandard inverted example} shows an example of an inverted semistandard tableau with $\lambda = (4,4,3)$ and $\mu = 1^1 2^1 3^1 4^1 5^2 6^1 7^2 8^2$.  If $n_{inv}(\tau) = i$, we refer to $\tau$ as an \textbf{i-inverted semistandard Young tableau}.  For given $\lambda$ and $\mu$, we denote the set of all such tableaux with precisely $i$ inversion pairs by $S_i(\lambda,\mu)$.  If we range across all possible numbers of inversions, we collectively refer to the set of all inverted semistandard tableaux as $ I(\lambda,\mu) = \bigcup_{i=0}^\infty S_i(\lambda,\mu)$.

\begin{figure}[h!]
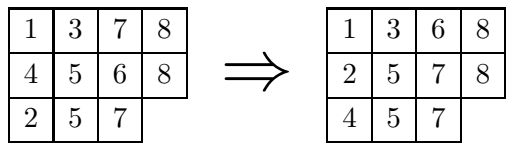

\centering
\begin{ytableau}
1 & 3 & 7 & 8 \\
4 & 5 & 6 & 8 \\
2 & 5 & 7
\end{ytableau}
\hspace{.08in}
\raisebox{-14pt}{\scalebox{2.5}{$\Rightarrow$}}
\hspace{.08in}
\begin{ytableau}
1 & 3 & 6 & 8 \\
2 & 5 & 7 & 8 \\
4 & 5 & 7
\end{ytableau}
\caption{An inverted semistandard tableau with inversion pairs $(2,4)$, $(3,5)$, $(6,7)$ and its standardization.}
\label{fig: semistandard inverted example}
\end{figure}

It should be noted that the notion of ``tableau inversion" presented here as well as in \cite{Fresse},\cite{BD} is a distinct concept from the ``inversions in standard Young tableaux" introduced by Shynar in \cite{Shynar}.  In \cite{Shynar}, a (weak) inversion in a standard Young tablau $T$ is a pair $(i,j)$ of entries such that $i < j$ and where $j$ appears both strictly south and strictly (resp. weakly) west of $j$ in $T$.  As such, Synar's inversions are a measure on standard Young tableau that do not address the more general row-standard case.  Although possibly related to the maximal possible number of inversion pairs in an inverted Young tableau $\tau$ with $st(\tau) = T$,  Shynar's distinct notion of tableau inversion will have absolutely no bearing on what follows.

\subsection{Outline of Results}
\label{subsec: outline of results}

We begin in Section \ref{sec: basic results} by generalizing a variety of basic results from \cite{BD} to the semistandard case.  Our most significant theorem in this realm is a closed formula giving the maximum possible number of inversion pairs for an inverted semistandard Young tableau of given shape and content, a result that is eventually appears as Theorem \ref{thm: maximum inversion number} and is presented in truncated form below:

\begin{theorem}
\label{thm: intro theorem 1}
Consider the shape $\lambda = (\lambda_1,\hdots,\lambda_m)$ and content $\mu = 1^{\mu_1} 2^{\mu_2} \hdots K^{\mu_K}$, and let $h_j$ be the height of the $j^{th}$ column in any tableau of shape $\lambda$.  If $I(\lambda,\mu)$ is nonempty, then the maximum number of inversions for any element of $I(\lambda,\mu)$ is:
\begin{center}
$\displaystyle{M_{\lambda,\mu} = \sum_j \binom{h_j}{2} - \sum_j \binom{\mu_i}{2}}$
\end{center}
\end{theorem}

Also included in Section \ref{sec: basic results} is the most theoretically significant result of the paper, a demonstration that the number of $i$-inverted semistandard Young tableaux of a fixed shape is invariant under ``permutation of content".  Eventually appearing as Theorem \ref{thm: permutation invariance, specific i}, notice that the simplified version shown below specializes to the well-known invariance of semistandard Young tableau presented in \cite{Fulton} and elsewhere if we let $i = 0$:

\begin{theorem}
\label{thm: intro theorem 2}
Take any shape $\lambda$ and any content $\mu$ using the entries $1,2,\hdots,M$.  For any permutation $\sigma$ on $M$ letters, we have $\vert S_i(\lambda,\mu) \vert = \vert S_i(\lambda,\sigma(\mu))\vert$ for all $i \geq 0$.
\end{theorem}

Section \ref{sec: enumeration} proceeds to give a series of direct enumerative results about inverted semistandard Young tableaux.  Closed formulas are given for the number of $i$-inverted tableaux in the one-column case (Theorem \ref{thm: enumeration one column}) and two-row case (Theorem \ref{thm: enumeration two row}), for any valid content $\mu$.  Our result for the one-column case, which reveals an intriguing new application of the q-factorial, is shown below:

\begin{theorem}
\label{thm: intro theorem 3}
Let $\lambda$ be the one-column tableau shape with $M$ total entries, and let $\mu = 1^{\mu_1} 2^{\mu_2} \hdots m^{\mu_m}$ be any content with $\sum_k \mu_k = M$.  Then the $\vert S_i(\lambda,\mu) \vert$ have generating function:
\begin{center}
$\displaystyle{\sum_{i=0}^\infty \vert S_i(\lambda,\mu) \vert q^i \ = \ \frac{[M]_q !}{[\mu_1]_q ! [\mu_2]_q ! \hdots [\mu_m]_q !}}$
\end{center}
Where $[p]_q = 1+q+\hdots+q^{p-1}$ is the q-number and $[p]_q! = [p]_q [p-1]_q \hdots [2]_q [1]_q$ is the q-factorial.
\end{theorem}

We close the paper with proofs of Theorems \ref{thm: enumeration i=1 rectangular} and \ref{thm: enumeration i=1 general}, which directly generalize results from \cite{BD} by demonstrating a bijection between $1$-inverted semistandard Young tableaux of a given shape of $0$-inverted semistandard Young tableaux of a collection of related shapes.  The more easily-digested specialization of this result to rectangular shapes is given below:

\begin{theorem}
\label{thm: intro theorem 4}
Let $m,n \geq 1$, and take the $m$-row tableau shapes $\lambda = (n,\hdots,n)$, $\widetilde{\lambda} = (n+1,n,\hdots,n,n-1)$.  Then $\vert S_1(\lambda,\mu) \vert = \vert S_0(\lambda,\mu) \vert$ for any content $\mu$ compatible with $\lambda$.
\end{theorem}

\section{Basic Results About Inverted Semistandard Young Tableaux}
\label{sec: basic results}

Before moving on to specific results about inverted semistandard Young tableaux, we formalize the definition of tableau inversion from Section \ref{sec: intro} via a complete order on the entries of any fixed column in a tableau.  So let $\tau$ be an inverted semistandard Young tableau.  Beginning with the rightmost column of $\tau$ and recursively working our way leftward, we place a complete order $\blacktriangleleft$ on the entries $\lbrace a_i \rbrace$ of each column as follows:

\begin{itemize}
\item If either $a_i$ or $a_j$ lacks an entry directly to its right and $a_i$ lies above $a_j$, then $a_i \blacktriangleleft a_j$.
\item If $a_i$ lies directly to the left of $b_i$, $a_j$ lies directly to the right of $b_j$, and $b_i < b_j$, then $a_i \blacktriangleleft a_j$.
\item If $a_i$ lies directly to the left of $b_i$, $a_j$ lies directly to the right of $b_j$, $b_i = b_j$, and $b_i \blacktriangleleft b_j$, then $a_i \blacktriangleleft a_j$.
\end{itemize}
We call $\blacktriangleleft$ the \textbf{height order} on the $j^{th}$ column of $\tau$.  If $c$ is the $k^{th}$ smallest element in its column of $\tau$ relative to the height order on that column, we say that $c$ has a \textbf{height} of $k$ in $\tau$ and write $ht(c)=k$.

The order $\blacktriangleleft$ tells us how a column of a tableau ``should be" ordered (relative to the column immediately on its right) if that column is to avoid any inversion pairs.  Notice that, if $\tau$ is column semistandard, then the height of $c$ is always equal to its row number.  In an inverted tableau, if $ht(c)$ does not equal the row number of $c$, then $c$ is involved in at least one in at least one inversion pair.  Most generally:

\begin{proposition}
\label{def: inversions from height order}
Let $\tau$ be a row-standard tableau and let $i,j$ be two entries from the same column of $\tau$.  Then $(i,j)$ forms an inversion pair of $\tau$ if and only if $i<j$ and $j \blacktriangleleft i$.
\end{proposition}

As one final basic comment about inversion pairs notice that, unlike in the non-repeated entry case of \cite{BD} and \cite{Fresse}, the location of an inversion pair $(i,j)$ is not uniquely identified by specifying which two entries are involved.  This is because, when one allows for repeated entries, it is possible for a pair of entries to appear together in more than one column of an inverted tableau.  When one needs to specify the column of origin for an inversion pair, we henceforth use $(i,j)^k$ to denote that the inversion pair $(i,j)$ occurs in the $k^{th}$ column of $\tau$.  Much as a standardization $T$ and a collection of inversion pairs was enough to uniquely identify a particular inverted standard Young tableau $\tau$ in the setting of \cite{Fresse} of \cite{BD}, it is straightforward to show that a standardization along with a collection of column-specified inversion pairs is enough to uniquely identify a particular inverted semistandard Young tableau.  However, this formal fact will not be necessary for anything that follows.

For the remainder of this section, we consider which basic results about inverted standard Young tableau from \cite{BD} generalize to the semistandard case.  The most fundamental result discussed in \cite{BD} was an explicit formula for the total number of inverted Young tableaux $\vert I(\lambda) \vert$ of an arbitrary shape $\lambda = (\lambda_1,\lambda_2,\hdots, \lambda_m)$.  A quick counting argument yielded:

\begin{equation}
\label{eq: total inverted tableaux, standard case}
\vert I(\lambda) \vert = \binom{\lambda_1 + \lambda_2 + \hdots + \lambda_m}{\lambda_m} \binom{\lambda_1 + \lambda_2 + \hdots + \lambda_{m-1}}{\lambda_{m-1}} \hdots \binom{\lambda_1 + \lambda_2}{\lambda_2} = \frac{(\lambda_1 + \lambda_2 + \hdots + \lambda_m)!}{\lambda_1 ! \lambda_2 ! \hdots \lambda_m !}
\end{equation}

Unfortunately, Equation \ref{eq: total inverted tableaux, standard case} does not appear to possess a tractable generalization to the general semistandard case $\vert I(\lambda,\mu) \vert$.  In particular, the necessity of the row-standard condition with regard to repeated entries prompts a series of increasingly sophisticated sub-cases and prevents a succinct probabilistic formulation akin to the rightmost side of Equation \ref{eq: total inverted tableaux, standard case}.  One of the few specific cases where $\vert I(\lambda,\mu) \vert$ is directly calculable with our current resources is when $\lambda$ has one column:

\begin{proposition}
\label{thm: total inversions, one column}
Let $\lambda = 1^M$ be the one-column tableau shape with $M$ total entries, and let $\mu = 1^{\mu_1} 2^{\mu_2} \hdots m^{\mu_m}$ be some content such that $\sum_k \mu_k = M$.  Then $\vert I(\lambda,\mu) \vert = \displaystyle{\frac{M!}{\mu_1 ! \mu_2 ! \hdots \mu_m !}}$.
\end{proposition}
\begin{proof}
Temporarily assume that all of the entries are distinct.  In this case there are $M!$ possible arrangements.  Dividing through by $\mu_i!$ then accounts for the fact that the $\mu_i$ instances of $i$ are indistinguishable, thus accounting for repetitions in our original enumeration.
\end{proof}

Luckily, the remaining results from Chapter 2 of \cite{BD} all admit generalizations to semistandard tableaux.  In Subsection \ref{subsec: max inversion number} we prove a general formula for the ``maximum inversion number" of an element in $I(\lambda,\mu)$.  In Subsection \ref{subsec: permutation invariance} we then prove an extremely useful result about the invariance of the $\vert S_i(\lambda,\mu) \vert$ under permutation of content: a theorem that has no analog in \cite{BD} but which directly generalizes the classic permutation invariance of (non-inverted) semistandard Young tableaux.  All enumerative results, including the ``straightforward cases" of one-column and two-row tableaux, are delayed until Section \ref{sec: enumeration} so that they can make direct usage of the permutation invariance guaranteed by Theorem \ref{thm: permutation invariance, specific i}.

\subsection{Maximum Number of Inversions for Shape $\lambda$ and Content $\mu$}
\label{subsec: max inversion number}

Obviously, a tableau of finite size cannot possess an infinite number of inversions.  It is then of interest to ask the maximum number of inversion pairs that an element of $I(\lambda,\mu)$ may possess.  In other words, what is the largest $i$ for which $\vert S_i(\lambda,\mu) \vert$ is nonempty?  For an inverted tableau without repeated entires, in \cite{BD} it was shown that the maximum such $i$ for an element of $I(\lambda)$ was:

\begin{equation}
\label{eq: maximum inversion number, standard case}
M_\lambda = \sum_j T_{(h_j - 1)} = \sum_j \binom{h_j}{2}
\end{equation} 

\noindent where $T_k = 1 + 2 + \hdots + k$ is the triangle number and $h_j$ is the height of the $j^{th}$ column in any tableau of shape $\lambda$.  In addition to an explicit formula for $M_\lambda$, \cite{BD} also showed that there was always precisely one element of $I(\lambda)$.  When one allows for general content $\mu$ with repeated entries, Equation \ref{eq: maximum inversion number, standard case} directly generalizes to the following:  

\begin{theorem}
\label{thm: maximum inversion number}
Let $\lambda = (\lambda_1,\hdots,\lambda_m)$ and $\mu = 1^{\mu_1} 2^{\mu_2} \hdots K ^{\mu_{K}}$, and define $h_j = \vert \lbrace \lambda_i \ \vert \ \lambda_i \geq j \rbrace \vert$ to be the height of the j\textsuperscript{th} column for any tableau of shape $\lambda$.  If $I(\lambda,\mu)$ is nonempty, then the maximum number of inversions for any inverted semistandard Young tableau of shape $\lambda$ with content $\mu$ is:

\begin{center}
$\displaystyle{M_{\lambda, \mu} = \sum_j T_{(h_j-1)} - \sum_i T_{(\mu_i -1)}} = \sum_j \binom{h_j}{2} - \sum_i \binom {\mu_i}{2}$
\end{center}

\noindent Moreover, this maximum inversion number is realized by precisely one inverted semistandard Young tableau of shape $\lambda$ and content $\mu$, so that $\vert S_{M_{\lambda,\mu}}(\lambda,\mu) \vert = 1$.

\end{theorem}
\begin{proof}
Our strategy is to construct an inverted tableau $\tau_{max}$ with precisely $\sum_j T_{h_j -1} - \sum_i T_{\mu_i - 1}$ inversion pairs and then argue why no other inverted tableau with the given $\lambda,\mu$ can have as many inversions than $\tau_{max}$.  To construct $\tau_{max}$ we work one column at a time, from right-to-left.  For the rightmost ($n^{th}$) column, we place the $h_n$ largest entries from top-to-bottom in the unique non-increasing order.  For the $(n-1)^{st}$ column, we work through the $h_{n-1}$ largest remaining remaining entries in decreasing order, placing each element in the available spot with the lowest height that does not violate the row-standard condition.  Notice that, if the instances of a repeated entry are split across two columns, this means that an entry need not be placed in the available spot with the lowest height.  Repeat this procedure for each of the remaining columns of $\tau_{max}$, placing the largest remaining entry in the lowest height slot available that does not result in identical entries being placed in the same row.  As we are recursively placing smaller entries leftward, the resulting tableau $\tau_{max}$ will always be row-standard.  For an example of the $\tau_{max}$ that results from this procedure, see Figure \ref{fig: maximum inversion number example}.

As seen in Equation \ref{eq: maximum inversion number, standard case}, if $\mu$ contains no repeated entries then $\tau_{max}$ has $M_{\lambda,\mu} = M_{\lambda} = \sum_j T_{h_j - 1}$ inversion pairs.  If $\mu$ has repeated entries, let $\widetilde{\tau}_{max}$ be the unique maximal inversion tableau of shape $\lambda$ with no repeated entries, as guaranteed by \cite{BD}.  Notice that re-indexing of repeated entries correlates instances of $i$ in $\tau_{max}$ to the set $\alpha_i = \lbrace \mu_1 + \hdots + \mu_{i-1} + 1, \hdots, \mu_1 + \hdots + \mu_{i-1} + \mu_i \rbrace$ in $\widetilde{\tau}_{max}$, although the placements of these sets of elements need not coincide because the ``preserving row-standard" condition in our construction of $\tau_{max}$ does not figure in the construction of $\widetilde{\tau}_{max}$.  See Figure \ref{fig: maximum inversion number example} for an example of the relationship between $\tau_{max}$ and $\widetilde{\tau}_{max}$.  Since we know $n_{inv}(\widetilde{\tau}_{max}) = \sum_j T_{h_j - 1}$, to show $n_{inv}(\tau_{max}) = \sum_j T_{h_j -1} - \sum_i T_{\mu_i - 1}$ we argue that $n_{inv}(\widetilde{\tau}_{max}) - n_{inv}(\tau_{max}) = \sum_i T_{\mu_i - 1}$.

\begin{figure}[h!]
\centering
\begin{ytableau}
1 & 4 & 5 \\
2 & 3 & 5 \\
2 & 3 \\
1 & 2
\end{ytableau}
\hspace{.5in}
\begin{ytableau}
2 & 7 & 10 \\
1 & 8 & 9 \\
3 & 6 \\
4 & 5
\end{ytableau}
\caption{The unique inverted semistandard Young tableau $\tau_{max}$ for $\lambda = (3,3,2,2)$, $\mu = 1^2 2^3 3^2 4^1 5^2$ with $M_{\lambda,\mu} = 7$ inversion pairs (left), and the related tableau $\widetilde{\tau}_{max}$ from the proof of Theorem \ref{thm: maximum inversion number} (right).}
\label{fig: maximum inversion number example}
\end{figure}

So fix an entry $i$ in $\tau_{max}$.  We compare inversion pairs in $\tau_{max}$ whose larger entry is $i$ with inversion pairs in $\widetilde{\tau}_{max}$ whose larger entry is an element of $\alpha_i$.  Notice that the instances of $i$ may or may not be split across two (adjacent) columns of $\tau_{max}$, but that the elements of $\alpha_i$ are split across two columns of $\widetilde{\tau}_{max}$ if and only if the instances of $i$ are split across two columns in $\tau_{max}$.  If instances of $i$ are split across two columns, we let $\zeta_1$ and $\zeta_2$ denote the subsets of those instances that lie in the leftward and rightward of the two columns, respectively.  Similarly, we let $\widetilde{\zeta}_1$ and $\widetilde{\zeta}_2$ denote the subsets of $\alpha_i$ that lie in the leftward and rightward of the two columns, respectively.  By construction, we always have $\vert \zeta_1 \vert = \vert \widetilde{\zeta}_1 \vert$ and $\vert \zeta_2 \vert = \vert \widetilde{\zeta}_2 \vert$.  In enumerating our inversion pairs, we consider two distinct cases:
\begin{enumerate}
\item If all instances of $i$ are in one column of $\tau_{max}$, by construction there exists a bijection between inversion pairs $(k,i)$ of $\tau_{max}$ with $k \neq i$ and inversion pairs $(\widetilde{k},\widetilde{i})$ of $\widetilde{\tau}_{max}$ with $\widetilde{i} \in \alpha_i$ and $\widetilde{k} \notin \alpha_i$.  In this case, $\widetilde{\tau}_{max}$ contains $T_{\mu_i}$ additional inversion pairs $(\widetilde{i_1},\widetilde{i_2})$ with $\widetilde{i_1},\widetilde{i_2} \in \alpha_i$ that have no analogue in $\tau_{max}$.  This follows from the fact that $(i,i)$ is not a valid inversion pair in $\tau_{max}$.
\item If instances of $i$ are split between two columns of $\tau_{max}$, let $\eta$ denote the set of entries in $\tau_{max}$ that lie directly to the left of elements of $\zeta_2$, and let $\widetilde{\eta}$ be the corresponding (reindexed) set of entries from $\widetilde{\tau}_{max}$.  By construction, every element of $\eta$ is less than $i$, and every element of $\widetilde{\eta}$ is less than every element of $\alpha_i$.  As $i$ is the smallest entry in the rightward of the two ``active" columns of $\tau_{max}$, $\tau_{max}$ cannot have an inversion pair of the form $(k,i)$ for any $k \in \eta$.  However, $\widetilde{\tau}_{max}$ will have an inversion pair of the form $(\widetilde{i_1},\widetilde{i_2})$ whenever $i_1 \in \widetilde{\eta}$ and $\widetilde{i_2} \in \widetilde{\zeta}_1$.  In this case, there exists a bijection between inversion pairs $(k,i)$ of $\tau_{max}$ with $k \neq i$ and inversion pairs $(\widetilde{k},\widetilde{i})$ with $\widetilde{i} \in \alpha_i$, $\widetilde{k} \notin \alpha_i \cup \widetilde{\eta}$.  Notice that there are precisely $\mu_i$ elements of $\zeta_1 \cup \eta$ in $\tau_{max}$, none of which may partake in any inversion pairs with other members of $\zeta_1 \cup \eta$, whereas $\widetilde{\zeta}_1 \cup \widetilde{\eta}$ is a ``fully-inverted" set in $\widetilde{\tau}_{max}$.  It follows that there exist precisely $T_{\mu_i}$ additional inversion pairs $(\widetilde{i_1},\widetilde{i_2})$ in $\widetilde{\tau}_{max}$ that have no analogue in $\tau_{max}$.  These are precisely the inversion pairs where $\widetilde{i_1},\widetilde{i_2} \in \widetilde{\zeta}_1 \cup \widetilde{\eta}$.
\end{enumerate}
In both cases, the number of inversion pairs in $\tau_{max}$ whose larger entry is $i$ is $T_{\mu_i}$ fewer than the number of inversion pairs in $\widetilde{\tau}_{max}$ whose larger entry is in $\alpha_i$.  Ranging over all distinct entries $i$ in $\tau_{max}$, we may conclude $n_{inv}(\widetilde{\tau}_{max}) - n_{inv}(\tau_{max}) = \sum_i T_{\mu_i - 1}$ and hence that $n_{inv}(\tau_{max}) = \sum_j T_{h_j -1} - \sum_i T_{\mu_i - 1}$.

It remains to be shown that no tableau in $I(\lambda,\mu)$ may have more inversions than $\tau_{max}$, as well as that $\tau_{max}$ is the unique element of $S_{M_{\lambda,\mu}}(\lambda,\mu)$.  First notice that $\tau_{max}$ has the property that every element in the $(j+1)^{st}$ column is at least as large as every element in the $j^{th}$ column, for all $j$.  If some other tableau $\tau \in I(\lambda,\mu)$ has entries $i_1 < i_2$ with some instance of $i_1$ in a column to the right of some instance of $i_2$, the number of inversions in $\tau$ whose larger entry is $i_2$ will be more than $T_{\mu_i}$ fewer than the number of inversion pairs in $\widetilde{\tau}_{max}$ whose larger entry is in $\alpha_{i_2}$.  It follows that we must have $n_{inv}(\tau) < n_{inv}(\tau_{max})$ for such a $\tau$.  This leaves $\tau \in I(\lambda,\mu)$ whose columns partition entries identically to $\tau_{max}$, but in which at least one of the columns has been ordered differently.  As $\tau_{max}$ was directly constructed to maximize the number of inversion pairs in each column (given the entries that must appear in that column), we clearly have $n_{inv}(\tau) < n_{inv}(\tau_{max})$ in this case.  It follows that $\tau_{max}$ is the unique element of $I(\lambda,\mu)$ with $M_{(\lambda,\mu)} = \sum_j T_{h_j -1} - \sum_i T_{\mu_i - 1}$ inversion pairs, and that no other element of $I(\lambda,\mu)$ may have more than $M_{\lambda,\mu}$ inversion pairs.
\end{proof}

\subsection{Invariance Under Permutation of Content}
\label{subsec: permutation invariance}

One of the most fundamental results involving semistandard Young tableaux is that the number of such tableaux with a fixed shape $\lambda$ is invariant under permutation of content.  In particular, given a content $\mu = (\mu_1,\mu_2,\hdots,\mu_M)$ and any permutation $\sigma \in S_M$, then $\vert S(\lambda,\mu) \vert = \vert S(\lambda,\sigma(\mu)) \vert$.  The most common proof of that fact, as outlined in \cite{Fulton}, identifies the number of semistandard tableaux of given content as the coefficient in a certain Schur polynomial and then utilizes the fact that Schur polynomials are symmetric polynomials.

In this Subsection we show that the ``permutation invariance" outlined above extends to i-inverted semistandard tableaux with a fixed number of inversions: that $\vert S_i(\lambda,\mu) \vert = \vert S_i(\lambda,\sigma(\mu)) \vert$ for all $i \geq 0$.  Since the traditional notion of semistandard Young tableau corresponds to the case of $i=0$, our general result specializes to the previously-established permutation invariance result of \cite{Fulton} when $i=0$.  Note that our techniques in no way reference symmetric polynomials, meaning that our $i=0$ specialization offers an apparently new proof of the well-known result from \cite{Fulton}.

Before proceeding to our primary proof, we require a series of technical lemmas characterizing how inversion numbers behave under manipulations of inverted semistandard Young tableaux with certain ``basic" shapes.

\begin{lemma}
\label{thm: permutation invariance lemma 1}
Let $\tau$ be a one-column row-standard tableau with $N$ total boxes and content $\mu = 1^j 2^{N-j}$.  If $\tau^*$ is the row-standard tableau of content $\mu$ obtained by reversing the vertical ordering of $\tau$, then $n_{inv}(\tau) + n_{inv}(\tau^*) = j(N-j)$.
\end{lemma}
\begin{proof}
Notice that the maximum possible number of inversions for a tableau with given $\lambda$ and $\mu$ is $j(N-j)$, occurring when all $2$ entries lie above all $1$ entries.  Now take any two entries $a_i,a_j$ in $\tau$ such that $a_i \neq a_j$.  The entries $a_i$ and $a_j$ form an inversion pair in $\tau$ if and only if their reflections $a_{N-i+1},a_{N-j+1}$ do not form an inversion pair in $\tau^*$.  It follows that any such pair $a_i,a_j$ constitutes an inversion pair in precisely one of $\tau$ or $\tau^*$.  Thus $n_{inv}(\tau) + n_{inv}(\tau^*) = j(N-j)$.
\end{proof}

\begin{lemma}
\label{thm: permutation invariance lemma 2}
Let $\tau$ be a one-column row-standard tableau with $N$ total boxes and content $\mu = 1^j 2^{N-j}$.  If $\bar{\tau}$ is the row-standard tableau of content $\bar{\mu} = 1^{N-j} 2^j$ obtained by flipping all instances of $1$ and $2$ in $\tau$, then $n_{inv}(\tau) + n_{inv}(\bar{\tau}) = j(N-j)$.
\end{lemma}
\begin{proof}
Notice that the maximum possible number of inversions for a tableau with given $\lambda$ and either content $\mu$ or $\bar{\mu}$ is $j(N-j)$, once again occurring when all $2$ entries lie above all $1$ entries.  Take any two entries $a_i,a_j$ in $\tau$ such that $a_i \neq a_j$, and let $\bar{a}_i,\bar{a}_j$ be the equivalently placed entries in $\bar{\tau}$.  The entries $a_i$ and $a_j$ form an inversion pair in $\tau$ if and only if $\bar{a}_i$ and $\bar{a}_j$ do not form an inversion pair in $\bar{\tau}$, as the relative ordering of the entries has been inverted in $\bar{\tau}$.  It follows that $n_{inv}(\tau) + n_{inv}(\bar{\tau}) = j(N-j)$.
\end{proof}

\begin{lemma}
\label{thm: permutation invariance lemma 3}
Let $\lambda$ be a two-column tableau shape $\lambda$ with $N$ total boxes.  If we define contents $\mu = 1^j 2^{N-j}$ and $\bar{\mu} = 1^{N-j} 2^j$, where $0 \leq j \leq N$, then $\vert S_i(\lambda,\mu) \vert = \vert S_i(\lambda,\bar{\mu}) \vert$ for every $i \geq 0$.
\end{lemma}
\begin{proof}
The general form of such a row-standard tableau (with content $\mu$ or $\bar{\mu}$) is shown in Figure \ref{fig: two-column lemma form}.  No matter the number of inversions, the only portion of such a tableau that is not determined is the one-column ``tail".  Observe that any inversion pairs from such a tableau must occur in its ``tail".  For any $\tau \in I(\lambda,\mu)$, we refer to the two-column ``head" subtableau as $\tau_+$ and the ``tail" subtableau as $\tau_-$.

Now fix $i \geq 0$, and define a map $\phi: S_i(\lambda,\mu) \rightarrow S_i(\lambda,\bar{\mu})$ that is the identity on $\tau_+$ and which maps each $\tau_-$ to $\bar{\tau}_-^*$.  Notice that the ``flipping" portion of $\phi \vert_{\tau_-}$ ensures that $\phi(\tau)$ has content $\bar{\mu}$.  By Lemmas \ref{thm: permutation invariance lemma 1} and \ref{thm: permutation invariance lemma 2} we see that $n_{inv}(\bar{\tau}^*) = j(N-j) - n_{inv}(\bar{\tau}) = j(N-j) - j(N-j) + n_{inv}(\tau) = n_{inv}(\tau)$, ensuring that $\phi(\tau)$ is in fact an element of $S_i(\lambda,\bar{\mu})$.  As $\phi$ is clearly reversible it represents a bijection.
\end{proof}

\begin{figure}[h!]
\centering
\scalebox{.85}{
\begin{ytableau}
1 & 2 \\
\vdots & \vdots \\
1/2 \\
\vdots
\end{ytableau}
}
\caption{General form a two-column row-standard tableau with content $\mu = 1^j 2^{N-j}$}
\label{fig: two-column lemma form}
\end{figure}

\begin{theorem}
\label{thm: permutation invariance, specific i}
Take any tableau shape $\lambda$ and any content $\mu = 1^{\mu_1}2^{\mu_2} \hdots M^{\mu_M}$ compatible with $\lambda$.  For any permutation $\sigma \in S_M$, we have $\vert S_i(\lambda,\mu) \vert = \vert S_i(\lambda,\sigma(\mu)) \vert$ for all $i \geq 0$.
\end{theorem}
\begin{proof}
We show $\vert S_i(\lambda,\mu) \vert = \vert S_i(\lambda,\sigma(\mu)) \vert$ for a simple transposition $(a,a+1) \in S_M$ of consecutive elements $a,a+1 \in \lbrace 1,2,\hdots,M \rbrace$.  The general result then follows from repeated application of our procedure.

So fix the inversion number $i \geq 0$ and consider two consecutive elements $a,a+1 \in \lbrace 1,2,\hdots M \rbrace$.  For any $\tau \in S_i(\lambda,\mu)$, identify the entries in the standardization $st(\tau)$ where $a,a+1$ appear.  The boxes with these two entries form a skew sub-tableau $\eta_\tau$ with content $a^{\mu_a}(a+1)^{\mu_{a+1}}$; as $st(\tau)$ is row-standard, no row in $\eta_\tau$ contains more than two entries.  Subdivide $\eta_\tau$ into a set of ``blocks" $\eta_\tau^j$, one for each upper-left corner entry $\alpha_j$ in $\eta_\tau$, by beginning with the first column of $\eta_\tau$ and assigning to $\eta_\tau^j$ all entries that are below or to the right of $\alpha_j$ and which have not yet been assigned to any previous block.  An example of this procedure is shown in Figure \ref{fig: standard inverted example}.

\begin{figure}[h!]
\centering
\scalebox{.8}{
\ytableaushort{\none \none \none \none \none \none D, \none \none \none \none C C, \none \none \none B B, \none \none \none B, A A, A A, A}
* {7,7,7,7,6,5,5}
* [*(gray!50)]{6+1,4+2,3+2,3+1,2,2,1}
}
\caption{A standarized 7-row tableau with four ``blocks" for the consecutive entries $a$,$a+1$.}
\label{fig: permutation invariance blocks}
\end{figure}

Looking more closely at our ``blocks", begin by noting that each block $\eta_\tau^j$ is a (non-skew) one- or two-column tableau with content of the form $a^x (a+1)^y$ for some $x,y \geq 0$.  Due to the way that all ``boundary entries" in $\eta_\tau$ are assigned to the leftmost of the two adjacent blocks (such as the lower-right corner in Figure \ref{fig: permutation invariance blocks} that is assigned a $B$ instead of a $C$), reassigning entries within a fixed $\eta_\tau^j$ never produces a new $\eta_\tau$ that fails to be row-standard.  The assignment of boundary entries also eliminates the possibility that reassigning entries within a fixed $\eta_\tau^j$ may produce a new inversion pair with one entry from each of two distinct blocks of $\eta_\tau$, as the boundary entry between two blocks is always necessary $a+1$.  As $a$ and $a+1$ are consecutive numbers, reassigning entries within $\eta_\tau^j$ never effects the relationship of entries within $\eta_\tau$ to  entries from outside of $\eta_\tau$.  This means that reassigning entries within a fixed $\eta_\tau^j$ never produces a $\tau$ that fails to be row-standard and doesn't create/eliminate inversions that involve an entry from outside of $\eta_\tau$.  The general conclusion to be drawn from all of these observations is that a valid reassignment within a fixed $\eta_\tau^j$ only effects $n_{inv}(\tau)$ in that it may create/eliminate inversion pairs where both involved entries are from that specific block $\eta_\tau^j$.

Now consider the set of all $i$-inverted tableaux in $S_i(\lambda,\mu)$.  Group $\tau \in S_i(\lambda,\mu)$ into subsets depending upon the exact shape (and placement) of the skew sub-tableaux $\eta_\tau$ within $st(\tau)$, and for each distinct placement $\gamma$ define $S_i^\gamma(\lambda,\mu) = \lbrace \tau \in S_i(\lambda,\mu) \ \vert \ \eta_\tau \text{ has placement } \gamma \text{ in } st(\tau) \rbrace$.  Similarly define $S_i^\gamma(\lambda,\sigma(\mu)) = \lbrace \tau \in S_i(\lambda,\sigma(\mu)) \ \vert \ \eta_\tau \text{ has placement } \gamma \text{ in } st(\tau) \rbrace$.  As $a$ and $a+1$ are consecutive numbers, notice that the acceptable placements $\gamma$ for $\eta_\tau$ are identical for contents $\mu$ and $\sigma(\mu)$.  Our strategy is to define a map $\phi: S_i(\lambda,\mu) \rightarrow S_i(\lambda,\sigma(\mu))$ that restricts to a bijection $\phi \vert_\gamma : S_i^\gamma (\lambda,\mu) \rightarrow S_i^\gamma(\lambda,\sigma(\mu))$ for each possible placement $\gamma$.  So take any placement $\gamma$ of $\eta_\tau$ and define $\phi \vert_\gamma: S_i^\gamma(\lambda,\mu) \rightarrow S_i^\gamma(\lambda,\sigma(\mu))$ as follows for an arbitrary $\tau \in S_i^\gamma(\lambda,\mu)$: 

\begin{enumerate}
\item For each block $\eta_\tau^j$ in $\eta_\tau$, we interpret $\eta_\tau^j$ as a one- or two-column inverted tableau of the sort described in Lemma \ref{thm: permutation invariance lemma 3} by considering the entries of $\eta_\tau^j$ along with all inversion pairs $(a,a+1)$ where both members are elements of $\eta_\tau^j$.  We then let the similarly placed block in $\phi \vert_\gamma(\tau)$ have entries and inversion pairs determined by the bijection of Lemma \ref{thm: permutation invariance lemma 3}.

\item Outside of $\eta_\tau$, $\phi \vert_\gamma$ is the identity apart from entries that lie directly to the left of a block $\eta_\tau^j$ where the bijection of Lemma \ref{thm: permutation invariance lemma 3} changes the height order.  In these circumstances, $\phi \vert_\gamma$ permutes the rows to the left of $\eta_\tau^j$ in the unique way that preserves the original height order of those elements.
\end{enumerate}

For an example of this map applied to a rectangular tableau with five total inversions, see Figure \ref{fig: permutation invariance example}.  In step \#1 of our procedure for $\phi$, Lemma \ref{thm: permutation invariance lemma 3} and our preceding ``block" observations guarantees that $\phi \vert_\gamma$ fixes the number of inversions involving entries from the block $\eta_\tau^j$ (or any other entries rightward from $\eta_\tau^j$).  In step \#2 of our procedure, the column reordering ensures that $\phi \vert_\gamma$ neither creates nor eliminates any inversion pairs involving entries that lie to the left of any block $\eta_\tau^j$.  The overall effect is that $\phi \vert_\gamma$ fixes the total number of inversions.  By the definition of the bijections from Lemma \ref{thm: permutation invariance lemma 3}, $\phi \vert_\gamma$ also clearly flips the content of $a$ and $a+1$.  It follows that $\phi \vert_\gamma$ is a well-defined map from $S_i^\gamma (\lambda,\mu)$ to $S_i^\gamma (\lambda,\sigma(\mu))$, and as $\phi \vert_\gamma$ is clearly reversible it is a bijection.  Thus $\vert S_i^\gamma (\lambda,\mu) \vert = \vert S_i^\gamma (\lambda,\sigma(\mu)) \vert$ for all placements $\gamma$, and we may conclude that $\vert S_i (\lambda,\mu) \vert = \vert S_i (\lambda,\sigma(\mu)) \vert$.
\end{proof}

\begin{figure}[h!]
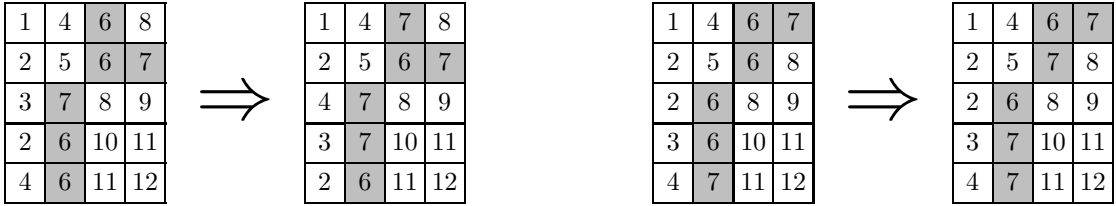

\centering
\scalebox{.9}{
\begin{ytableau}
1 & 4 & *(gray!50) 6 & 8 \\
2 & 5 & *(gray!50) 6 & *(gray!50) 7 \\
3 & *(gray!50) 7 & 8 & 9 \\
2 & *(gray!50) 6 & 10 & 11 \\
4 & *(gray!50) 6 & 11 & 12
\end{ytableau}
\hspace{.08in}
\raisebox{-32pt}{\scalebox{3}{$\Rightarrow$}}
\hspace{.08in}
\begin{ytableau}
1 & 4 & *(gray!50) 7 & 8 \\
2 & 5 & *(gray!50) 6 & *(gray!50) 7 \\
4 & *(gray!50) 7 & 8 & 9 \\
3 & *(gray!50) 7 & 10 & 11 \\
2 & *(gray!50) 6 & 11 & 12
\end{ytableau}
\hspace{1in}
\begin{ytableau}
1 & 4 & *(gray!50) 6 & *(gray!50) 7 \\
2 & 5 & *(gray!50) 6 & 8 \\
2 & *(gray!50) 6 & 8 & 9 \\
3 & *(gray!50) 6 & 10 & 11 \\
4 & *(gray!50) 7 & 11 & 12
\end{ytableau}
\hspace{.08in}
\raisebox{-32pt}{\scalebox{3}{$\Rightarrow$}}
\hspace{.08in}
\begin{ytableau}
1 & 4 & *(gray!50) 6 & *(gray!50) 7 \\
2 & 5 & *(gray!50) 7 & 8 \\
2 & *(gray!50) 6 & 8 & 9 \\
3 & *(gray!50) 7 & 10 & 11 \\
4 & *(gray!50) 7 & 11 & 12
\end{ytableau}
}
\caption{Part of $\phi: S_5(\lambda,\mu) \rightarrow S_5(\lambda,\sigma(\mu))$ for $\lambda = (4,4,4,4,4)$, $\mu = 1^1 2^2 3^1 4^2 5^1 6^4 7^2 8^2 9^1 10^1 11^2 12^1$ and $\sigma= (6 \ 7)$.  The left side shows $\tau \in S_5(\lambda,\mu)$ and $\sigma(\tau)$, both of which have inversion pairs $(3,4)^1$, $(4,5)^2$, $(6,7)^2$, $(6,7)^2$, $(7,8)^4$.  The right side shows $st(\tau)$ and $st(\sigma(\tau))$.  Notice how the first column of $\sigma(\tau)$ has been reordered to the left of the bottom ``block" in order to maintain the relative height order.}
\label{fig: permutation invariance example}
\end{figure}

As one quick corollary of Theorem \ref{thm: permutation invariance, specific i}, notice that the total number of inverted semistandard tableaux of shape $\lambda$ (ranging over all possible numbers of inversions) is invariant under permutation of content:

\begin{corollary}
\label{thm: permutation invariance, total number}
Take any tableau shape $\lambda$ and any content $\mu = 1^{\mu_1}2^{\mu_2} \hdots K^{\mu_K}$ compatible with $\lambda$.  For any permutation $\sigma \in S_M$, $\vert I(\lambda,\mu) \vert = \vert I(\lambda,\sigma(\mu)) \vert$.
\end{corollary}

\section{Enumeration of Inverted Semistandard Young Tableaux}
\label{sec: enumeration}

With the tools of Section \ref{sec: basic results} in place, we are ready to present enumerative results about inverted semistandard Young tableaux.  As with \cite{BD}, closed formulas for $\vert S_i(\lambda,\mu) \vert$ when $i$ is arbitrary are only tractable for certain ``easy" choices of $\lambda$, namely one-column and one- or two-row shapes.  After fully addressing those ``easy" cases, we directly enumerate $S_1(\lambda,\mu)$ for arbitrary $\lambda$ and $\mu$ by placing that set in bijection with a collection of sets of (non-inverted) semistandard Young tableaux $\bigcup_k S_0(\widetilde{\lambda}_k,\mu)$, thus generalizing Theorems 3.1 and 3.2 of \cite{BD}.

\subsection{Enumerating $i$-Inverted Semistandard Young Tableaux, $\lambda = (1,1,\hdots,1)$}
\label{subsec: enumeration one column}

In the standard tableaux setting of \cite{BD}, one of the few choices of $\lambda$ for which specific $\vert S_i (\lambda) \vert$ could be directly computed were the single-column shapes $\lambda = 1^m$ ($m \geq 1$).  As inverted standard Young tableaux with one-column are equivalent to permutations, that paper cited the standard result \cite{Stanley2} to give $\vert S_i (\lambda) \vert = M(m-1,i)$, where $M(m-1,i)$ is the Mahonian number.  If we let $[p]_q = 1 + q + \hdots + q^{p-1}$ be the q-number and let $[p]_q! = [1]_q [2]_q \hdots [p]_q$ be the q-factorial, this gave the generating function:
\begin{equation}
\label{eq: one-column generating function, standard case}
\sum_{i=0}^\infty \vert S_i(\lambda) \vert q^i \ = \ (1+q)(1+q+q^2) \hdots (1+q+ \hdots + q^{m-1}) \ = \ [2]_q [3]_q \hdots [m]_q \ = \ [m]_q !
\end{equation}

The single-column case is also relatively tractable when we allow for repeated entries, yielding a direct generalization of the generating function from Equation \ref{eq: one-column generating function, standard case}.  In what follows, we use the standard notation $\binom{a}{b}_q = \frac{[a]_q !}{[b]_q ! [a-b]_q !} = \frac{(1-q^a)(1-q^{a-1}) \hdots (1-q^{a-b+1})}{(1-q)(1-q^2) \hdots (1-q)^b}$ for the q-binomial coefficients.

\begin{theorem}
\label{thm: enumeration one column}
Let $\lambda = 1^M$ be the one-column tableau shape with $M$ total entries, and let $\mu = 1^{\mu_1} 2^{\mu_2} \hdots M^{\mu_m}$ be some content such that $\sum_k \mu_k = M$.  Then we have generating function:
\begin{center}
$\displaystyle{\sum_{i=0}^\infty \vert S_i(\lambda,\mu) \vert q^i \ = \ \binom{\mu_1}{\mu_1}_q \binom{\mu_1+\mu_2}{\mu_2}_q \binom{\mu_1 + \mu_2 + \mu_3}{\mu_3}_q \hdots \binom{M}{\mu_m}_q \ = \ \frac{[M]_q !}{[\mu_1]_q ! [\mu_2]_q ! \hdots [\mu_m]_q !}}$
\end{center}
\end{theorem}
\begin{proof}
It is possible to ``build up" any inverted tableau $\tau \in I(\lambda,\mu)$ by recursively inserting $\mu_n$ copies of $n$ into a one-column tableau $\tau_{n-1}$ with content $1^{\mu_1} 2^{\mu_2} \hdots (n-1)^{\mu_{n-1}}$, producing a sequence $\lbrace \tau_1, \tau_2, \hdots, \tau_m \rbrace$ of one-column tableaux such that $\tau_m = \tau$.  Notice that distinct placements at any step in this process always results in distinct $\tau$, so this procedure describes a way to uniquely determine every element of $I(\lambda,\mu)$.  Furthermore, the number of inversion pairs in the resulting $\tau$ whose larger element is $n$ is determined entirely by the insertion of the $\mu_n$ copies of $n$ into $\tau_{n-1}$: this ``level $n$ step" is the only point at which such inversion pairs may appear, and the number of such inversion pairs is not dependent upon the prior arrangement of the $1^{\mu_1} 2^{\mu_2} \hdots (n-1)^{\mu_{n-1}}$ or the later placement of larger entries.

So fix a one-column tableau $\tau_{n-1}$ with content $1^{\mu_1} 2^{\mu_2} \hdots (n-1)^{\mu_{n-1}}$.  A copy of $n$ placed above $j$ entries in $\tau_{n-1}$ results in $j$ inversion pairs whose larger entry is that instance of $n$.  In particular, each instance of $n$ may be involved in up to $\mu_1 + \mu_2 + \hdots + \mu_{n-1}$ inversion pairs where it is the larger entry, and the number of such inversion pairs involving a particular instance of $n$ is independent of the placement of other instances of $n$.  Now consider the number of tableaux obtained from $\tau_{n-1}$ with precisely $i$ inversion pairs whose larger entry is $n$.  By our preceding comments, these tableaux are in bijection with partitions of $i$ into at most $\mu_n$ parts (one part corresponding to each instance of $n$) where each part is less than or equal to $\mu_1 + \hdots + \mu_{n-1}$.

It is well known that the coefficient of $q^i$ in $\binom{a+b}{a}_q$ equals the number of partitions of $i$ into at most $a$ parts, with each part less than or equal to $b$.  If $\vert \tau_{n-1}^i \vert$ denotes the number tableaux $\tau_n$ obtained from $\tau_{n-1}$ with precisely $i$ inversion pairs whose larger entry is $n$, we then have generating function

\begin{equation}
\label{eq: single column generating function lemma}
\displaystyle{\sum_{i=0}^\infty \vert \tau_{n-1}^i \vert q^i = \binom{\mu_1 + \mu_2 + \hdots + \mu_n}{\mu_n}_q}
\end{equation}

As every inversion pair in $\tau \in I(\lambda,\mu)$ appears at a unique step in the sequence $\lbrace \tau_1, \tau_2, \hdots , \tau_m \rbrace$, multiplying the generating functions of Equation \ref{eq: single column generating function lemma} for $1 \leq i \leq m$ gives the result.
\end{proof}

Notice that if $\mu_k = 1$ for all $k$, Theorem \ref{thm: enumeration one column} recovers the standard tableaux result of Equation \ref{eq: one-column generating function, standard case}.  Specialization of Theorem \ref{thm: enumeration one column} at $q=1$ shows that the total number of inverted semistandard Young tableaux of shape $\lambda$ is $\vert I(\lambda,\mu) \vert = \frac{M!}{\mu_1 ! \mu_2 ! \hdots \mu_m !}$, independently verifying Proposition \ref{thm: total inversions, one column}.

\subsection{Enumerating $i$-Inverted Semistandard Young Tableaux, $\lambda = (n,n)$}
\label{subsec: enumeration two rows}

The other basic shapes $\lambda$ that admitted a direct enumeration of i-inverted standard Young tableaux in \cite{BD} were one-row and two-row tableau shapes.  The row-standard condition made one-row shapes predictably trivial: if $\lambda = (n)$ for any $n \geq 1$, then $\vert S_0(\lambda) \vert =1$ and $\vert S_i(\lambda) \vert = 0$ for all $i \geq 1$.  The two-row case involved a recognition of the fact that any two-row inverted tableau necessarily ``split" after a column in which it possessed an inversion pair.  As shown in \cite{BD}, if $\lambda = (n,n)$ the formula for $\vert S_i (\lambda) \vert$ depends upon summations of products of Catalan numbers $C_k$ where the subscripts in each term partition $n$:

\begin{equation}
\label{eq: two-row enumeration, standard case}
\vert S_i(\lambda) \vert \ = \ \left( \sum_{k_1 + \hdots + k_i = n} C_{k_1} C_{k_2} \hdots C_{k_i} \right) + \left( \sum_{l_1 + \hdots + l_{i+1} = n} C_{l_1} C_{l_2} \hdots C_{l_{i+1}} \right)
\end{equation}

Equation \ref{eq: two-row enumeration, standard case} admits a very direct generalization to the semistandard case in the form of Theorem \ref{thm: enumeration two row}.  This theorem also marks our first usage of Theorem \ref{thm: permutation invariance, specific i} as a powerful simplifying tool:

\begin{theorem}
\label{thm: enumeration two row}
Let $\lambda = (n,n)$, any $n \geq 1$, and let $\mu= 1^{\mu_1} 2^{\mu_2} \hdots$ be some content such that $\sum_k \mu_k = 2n$.

\begin{enumerate}
\item If $\mu_k > 2$ for any $k$, then $\vert S_i(\lambda,\mu) \vert = 0$ for all $i \geq 0$.
\item If $\mu_k = 2$ for precisely $m$ choices of $k$ and $\mu_k = 1$ for the remaining $2n-2m$ choices of $k$, then:
\begin{center}
$\displaystyle{\vert S_i(\lambda,\mu) \vert = \left( \sum_{j_1 + \hdots + j_i = n-m} C_{j_1} C_{j_2} \hdots C_{j_i} \right)  + \left( \sum_{l_1 + \hdots + l_{i+1} = n-m} C_{l_1} C_{l_2} \hdots C_{l_{i+1}} \right) }$
\end{center}
Where $C_j$ is the $j^{th}$ Catalan number, and the summations run over all ordered partitions of length $i$ and $i+1$, respectively.
\end{enumerate}
\end{theorem}
\begin{proof}
Case \#1 is immediate because no such tableau can be row-standard.  For Case \# 2, by Theorem \ref{thm: permutation invariance, specific i} we may assume that $\mu_k = 2$ for $1 \leq k \leq m$ and $\mu_k = 1$ for $k > m$.  This means that the first $m$ columns of any $\tau \in  I(\lambda,\mu)$ each consist of two instances of the same entry, and hence cannot partake in an inversion pair.  Thus the only place where $\tau$ isn't predetermined, as well as the only place where $\tau$ may possess inversion pairs, is over it's final $n-m$ columns.  Notice that, when restricted to these final $n-m$ columns, any $\tau \in I(\lambda,\mu)$ becomes a inverted standard Young tableau with $2(n-m)$ distinct entries $m+1,m+2,\hdots,2n-m$.  If we define $\widetilde{\lambda} = (n-m,n-m)$, a truncation of $\tau \in I(\lambda,\mu)$ to its final $n-m$ columns and then a reindexing of its entries yields a bijection between $S_i(\lambda,\mu)$ and $S_i(\widetilde{\lambda})$ for all $i \geq 0$.  The result then follows from Theorem 2.3 of \cite{BD}.
\end{proof}

If $\mu_k = 1$ for all $k$, the formula of Theorem \ref{thm: enumeration two row} very obviously simplifies to the standard tableaux formula of Equation \ref{eq: two-row enumeration, standard case}.  Less obvious from Theorem \ref{thm: enumeration two row} is an enumeration of the total number of inverted tableaux $\vert I(\lambda,\mu) \vert$, but a similar bijection with shorter two-row standard tableaux yields the following:

\begin{proposition}
\label{thm: total inversions two row}
Let $\lambda = (n,n)$ for any $n \geq 1$, and let $\mu = 1^{\mu_1}2^{\mu_2} \hdots$ be some content with $\sum_k \mu_k = 2n$.  If $\mu_k = 2$ for precisely $m$ choices of $k$ ($0 \leq m \leq n$) and $\mu_k =1$ otherwise, then $\vert I(\lambda,\mu) \vert = \binom{2(n-m)}{n-m}$.
\end{proposition}
\begin{proof}
By Theorem \ref{thm: permutation invariance, total number}, we may assume that $\mu_k =2$ for $1 \leq k \leq m$ and $\mu_k = 1$ for $k > m$.  As in the proof of Theorem \ref{thm: enumeration two row}, this means that the first $m$ columns of any $\tau \in I(\lambda,\mu)$ each consist of two instances of the same entry, and that elements of $I(\lambda,\mu)$ are in bijection with inverted standard Young tableaux of shape $\widetilde{\lambda} = (n-m,n-m)$.  The result then follows from Proposition 2.1 of \cite{BD}, as we merely need to specify which of the $2(n-m)$ non-repeated entries appear in the first row.
\end{proof}

\subsection{Enumerating $1$-Inverted Semistandard Young Tableaux}
\label{subsec: enumeration i=1}

Enumeration of $S_i(\lambda,\mu)$ for general $\lambda$ and any $i \geq 0$ is a daunting (and potentially intractable) task that wasn't even accomplished in the non-repeated entry case of \cite{BD}.  If one wishes to address arbitrary $\lambda$, one enumeration $\vert S_i (\lambda,\mu) \vert$ that remains approachable is the single inversion pair case of $i=1$.  In this subsection we exhibit a bijection between 1-inverted semistandard Young tableaux and (0-inverted) semistandard Young tableaux of a collection of related shapes.  Unlike in \cite{BD}, this doesn't allow for an immediate determination of $\vert S_1(\lambda,\mu) \vert$ via the hook-length formula, seeing as the hook-content formula for semistandard Young tableaux is ill-suited to enumeration of tableaux with a specific content.  Nonetheless, it does replace the set $S_1(\lambda,\mu)$ with a far better understood collection of sets $S_0(\widetilde{\lambda},\mu)$.  The author also conjectures that the approach of Theorems \ref{thm: enumeration i=1 rectangular} and \ref{thm: enumeration i=1 general} may be modified for the calculation of $\vert S_i (\lambda,\mu) \vert$ for some $i > 1$, akin to Conjecture 4.4 in \cite{BD}.

\begin{theorem}
\label{thm: enumeration i=1 rectangular}
Let $m,n \geq 1$, and consider the m-row shapes $\lambda = (n,\hdots,n), \widetilde{\lambda} = (n+1,n,\hdots,n,n-1)$.  For any content $\mu$ compatible with $\lambda$, $\vert S_1(\lambda,\mu) \vert = \vert S_0(\widetilde{\lambda},\mu) \vert$.  
\end{theorem}

If $\lambda$ is a rectangular shape of size $m \times n$, then $\widetilde{\lambda}$ is the ``stair-step" shape formed by moving the the lower-right corner in the Young diagram of shape $\lambda$ to a new, $(n+1)^{st}$ column.  See Figure \ref{fig: rectangular to stair-step comparison} for an example of this shape change.  Also pause to notice that this is the only way to rearrange the boxes in a Young diagram of shape $\lambda$ to produce another valid Young diagram whereby the old lower-right corner in $\lambda$ has been moved to a new lower-right corner in a higher row.

\begin{figure}[h!]
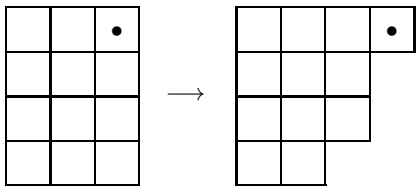

\centering
\begin{footnotesize}
\ydiagram{3,3,3,3}
*[\bullet]{0,0,0,2+1}
\hspace{.1in}
\raisebox{-18.5pt}{$\longrightarrow$}
\hspace{.08in}
\ydiagram{4,3,3,2}
*[\bullet]{3+1}
\end{footnotesize}
\caption{Shape change in the bijection of Theorem \ref{thm: enumeration i=1 rectangular} for $\lambda = (3,3,3,3)$}
\label{fig: rectangular to stair-step comparison}
\end{figure}

\begin{proof}[Proof of Theorem \ref{thm: enumeration i=1 rectangular}]
As in Theorem 3.1 of \cite{BD}, we define two functions $\phi_1: S_1(\lambda,\mu) \rightarrow S_0(\widetilde{\lambda},\mu)$, $\phi_2: S_0(\widetilde{\lambda},\mu) \rightarrow S_1(\lambda,\mu)$ such that $\phi_1$ and $\phi_2$ are inverses of one another.  The general outline of both procedures is in line with the ``repeated bumping" maps defined in \cite{BD}, apart from the addition of conditions that address how bumping behaves in the vicinity of repeated entries.

For the map $\phi_1: S_1(\lambda,\mu) \rightarrow S_0(\widetilde{\lambda},\mu)$, take $\tau \in S_1(\lambda,\mu)$ and identify the sole inversion pair $(a,b)$ of $\tau$.  Assume that $(a,b)$ appears in the $k^{th}$ column of $\tau$; as $(a,b)$ is the tableau's only inversion pair it must be the case that $b$ appears immediately above $a$ in the $k^{th}$ column.  Let $b = c_k$.  Our strategy is to recursively ``bump" a sequence of elements $\lbrace c_k,c_{k+1},\hdots,c_n \rbrace$, one from each column of $\tau$ beginning with the $k^{th}$ column, rightward by one column each.  Our procedure is defined as follows:

\begin{enumerate}
\item Beginning at the site of the inversion pair $(a,b)$ in the $k^{th}$ column, let $c_k = b$.  If there are any columns in $\tau$ to the left of the $k^{th}$, reorder those columns so that they are each non-increasing from top-to-bottom.  This reordering guarantees that no new inversion pairs will be added in leftward columns due to the elimination of the $(a,b)$ inversion pair.
\item At the $j^{th}$ column in the procedure, if $j<n$ let $c_{j+1}$ be the smallest entry in the $(j+1)^{st}$ column such that $c_{j+1} > c_j$.  Then move $c_j$ to the spot occupied by $c_{j+1}$, temporarily allowing two entries in that spot.  This leaves an empty box in the $j^{th}$ column where $c_j$ was formerly located.  Recursively fill any open spots in the $j^{th}$ column by moving the smaller of the two entries directly below or directly to the right of the empty box into that box.  If both entries are equal at any step in this process, move the entry directly to the right of the empty box into that box.  Repeat this procedure until the empty box has been moved into the $(j+1)^{st}$ column.
\item At the $j^{th}$ column in the procedure, if $j=n$ move $c_n$ to the top spot of a new $(n+1)^{st}$ column.  Then slide all entries that were below $c_n$ in the $n^{th}$ column up by one spot.
\end{enumerate}

An example of the full procedure is shown in Figure \ref{fig: rectangular i=1 to i=0 example}.  Notice that this procedure always results in a tableau of the correct shape $\widetilde{\lambda}$, and that the resulting tableau $\widetilde{\lambda}$ lacks inversion pairs because the procedure is defined to ensure that every column is non-increasing from top-to-bottom.  The ``forward bumping" and ``back sliding" procedures are also defined to ensure that the resulting tableau is row-standard.  In particular, observe that the condition in \#2 whereby the rightward of two identical entries is slid left prevents two identical entries from appearing in the same row.  As every step in the procedure is uniquely determined, we may conclude that the procedure defines a well-defined function $\phi_1: S_1(\lambda,\mu) \rightarrow S_0(\widetilde{\lambda},\mu)$.

\begin{figure}[h!]
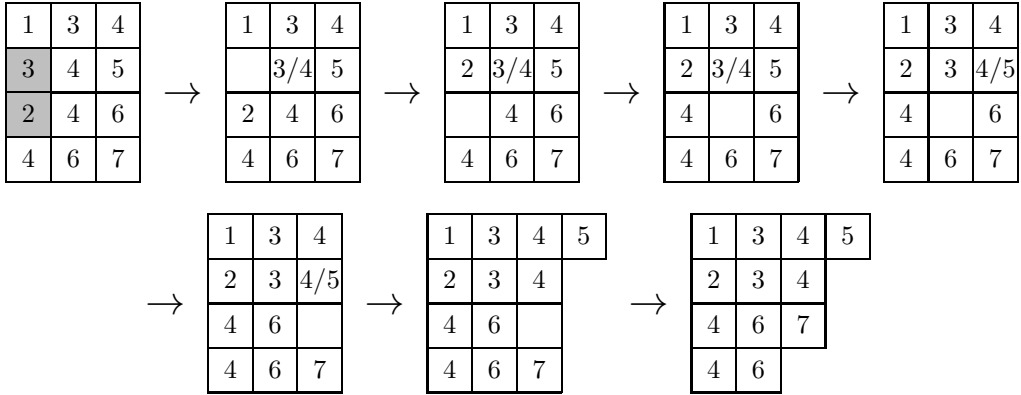

\centering
\begin{small}
\ytableausetup{boxsize=1.65em}
\begin{ytableau}
1 & 3 & 4 \\
*(gray!50) 3 & 4 & 5 \\
*(gray!50) 2 & 4 & 6 \\
4 & 6 & 7
\end{ytableau}
\hspace{.03in}
\raisebox{-21pt}{\scalebox{1.5}{$\rightarrow$}}
\hspace{.03in}
\begin{ytableau}
1 & 3 & 4 \\
& 3/4 & 5 \\
2 & 4 & 6 \\
4 & 6 & 7
\end{ytableau}
\hspace{.03in}
\raisebox{-21pt}{\scalebox{1.5}{$\rightarrow$}}
\hspace{.03in}
\begin{ytableau}
1 & 3 & 4 \\
2 & 3/4 & 5 \\
 & 4 & 6 \\
4 & 6 & 7
\end{ytableau}
\hspace{.03in}
\raisebox{-21pt}{\scalebox{1.5}{$\rightarrow$}}
\hspace{.03in}
\begin{ytableau}
1 & 3 & 4 \\
2 & 3/4 & 5 \\
4 &  & 6 \\
4 & 6 & 7
\end{ytableau}
\hspace{.03in}
\raisebox{-21pt}{\scalebox{1.5}{$\rightarrow$}}
\hspace{.03in}
\begin{ytableau}
1 & 3 & 4 \\
2 & 3 & 4/5 \\
4 & & 6 \\
4 & 6 & 7
\end{ytableau}

\vspace{.15in}

\raisebox{-21pt}{\scalebox{1.5}{$\rightarrow$}}
\hspace{.03in}
\begin{ytableau}
1 & 3 & 4 \\
2 & 3 & 4/5 \\
4 & 6 & \\
4 & 6 & 7
\end{ytableau}
\hspace{.03in}
\raisebox{-21pt}{\scalebox{1.5}{$\rightarrow$}}
\hspace{.03in}
\begin{ytableau}
1 & 3 & 4 & 5\\
2 & 3 & 4 \\
4 & 6 & \\
4 & 6 & 7
\end{ytableau}
\hspace{.03in}
\raisebox{-21pt}{\scalebox{1.5}{$\rightarrow$}}
\hspace{.03in}
\begin{ytableau}
1 & 3 & 4 & 5\\
2 & 3 & 4 \\
4 & 6 & 7\\
4 & 6
\end{ytableau}
\end{small}
\caption{Part of the $S_1(\lambda,\mu) \hookrightarrow S_0(\widetilde{\lambda},\mu)$ bijection for $\lambda = (3,3,3,3)$ and $\mu = 1^1 2^1 3^2 4^4 5^1 6^2 7^1$}
\label{fig: rectangular i=1 to i=0 example}
\end{figure}

For the map $\phi_2 : S_0(\widetilde{\lambda},\mu) \rightarrow S_1(\lambda,\mu)$, take $T \in S_0(\widetilde{\lambda},\mu)$ and define $c_n$ to be the sole entry in the $n^{th}$ column of $T$.  Our strategy is to define a sequence of entries $\lbrace c_n, c_{n-1}, \hdots \rbrace$, where $c_j$ begins in the $(j+1)^{st}$ column of $T$, and then recursively ``reverse bump" $c_j$ into the $j^{th}$ column in a manner that reverses the $\phi_1$ procedure between any two columns.  Our new procedure at the $j^{th}$ column is as follows:

\begin{enumerate}
\item Consider $c_j$, which begins in the $(j+1)^{st}$ column.  There is guaranteed to be an empty box in the $j^{th}$ column.  Recursively fill that empty box with the largest entry from among $c_j$, the entry directly above the empty box, and the entry directly to the left of the empty box.  If the largest value at any point in this process is shared by two or more of those three entries, preference is given to entries that begin in a more leftward position.  Repeat this procedure until the empty box is moved leftward to the $(j-1)^{st}$ column or $c_j$ directly fills the empty box.
\item If the empty box in the $j^{th}$ column is moved leftward to the $(j-1)^{st}$ column, define $c_{j-1}$ to be the largest entry in the $j^{th}$ that is strictly less than $c_j$.  Move $c_j$ in the box occupied by $c_{j-1}$, temporarily producing a box with two entries, and then repeat step \#1 with $c_{j-1}$ and the $(j-1)^{st}$ column.
\item If the empty box in the $j^{th}$ column is directly filled by $c_j$, we introduce a single new inversion pair $(a,c_j)$ in the $j^{th}$ column with $c_j$ and the entry lying directly above the box into which $c_j$ was inserted.  Do this by flipping the rows containing $c_j$ and $a$ from the $j^{th}$ column leftward.  Flipping leftward entries along with $c_j$ and $a$ ensures that no additional inversion pairs are added in leftward columns.
\end{enumerate}

For an example of this second procedure, see Figure \ref{fig: rectangular i=0 to i=1 example}.  Notice that the ``sliding" and ``reverse bumping" rules of steps \#1 and \#2, along with the addendum addressing when identical entries are involved, ensure that the tableau is row-standard and column-semistandard.  This means that the only inversion in the resulting tableau is the one introduced in step \#3. Also notice that the resulting tableaux always admits an inversion at this spot in the $j^{th}$ column since $c_j < c_{j+1}$ and our procedure ensures that the entry $c_j$ always drops down at least one row at this final step (see the proof of Theorem 3.1 in \cite{BD} for additional discussion of this final fact).  As every step in the procedure is uniquely determined, we conclude that the procedure induces a well-defined function $\phi_2: S_0(\widetilde{\lambda},\mu) \rightarrow S_1(\lambda,\mu)$.

\begin{figure}[h!]
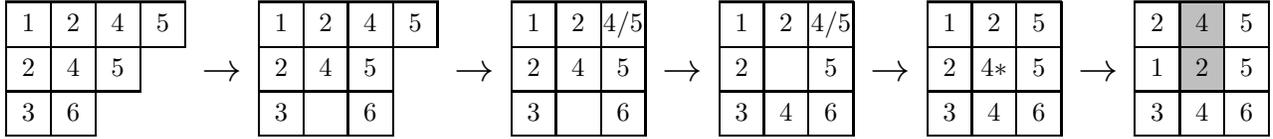

\centering
\begin{small}
\ytableausetup{boxsize=1.65em}
\begin{ytableau}
1 & 2 & 4 & 5 \\
2 & 4 & 5 \\
3 & 6
\end{ytableau}
\hspace{.0in}
\raisebox{-13pt}{\scalebox{1.5}{$\rightarrow$}}
\hspace{.0in}
\begin{ytableau}
1 & 2 & 4 & 5 \\
2 & 4 & 5 \\
3 & & 6
\end{ytableau}
\hspace{.0in}
\raisebox{-13pt}{\scalebox{1.5}{$\rightarrow$}}
\hspace{.0in}
\begin{ytableau}
1 & 2 & 4/5 \\
2 & 4 & 5 \\
3 & & 6
\end{ytableau}
\hspace{.0in}
\raisebox{-13pt}{\scalebox{1.5}{$\rightarrow$}}
\hspace{.0in}
\begin{ytableau}
1 & 2 & 4/5 \\
2 &  & 5 \\
3 & 4 & 6
\end{ytableau}
\hspace{.0in}
\raisebox{-13pt}{\scalebox{1.5}{$\rightarrow$}}
\hspace{.0in}
\begin{ytableau}
1 & 2 & 5 \\
2 & 4* & 5 \\
3 & 4 & 6
\end{ytableau}
\hspace{.0in}
\raisebox{-13pt}{\scalebox{1.5}{$\rightarrow$}}
\hspace{.0in}
\begin{ytableau}
2 & *(gray!50) 4 & 5 \\
1 & *(gray!50) 2 & 5 \\
3 & 4 & 6
\end{ytableau}
\end{small}
\caption{Part of the $S_0(\widetilde{\lambda},\mu) \hookrightarrow S_1(\lambda,\mu)$ bijection for $\lambda = (3,3,3), \mu = 1^1 2^2 3^1 4^2 5^2 6^1$.  Notice how the first column has been reordered at the final step to preserve its height ordering.}
\label{fig: rectangular i=0 to i=1 example}
\end{figure}

Our maps $\phi_1$ and $\phi_2$ have been constructed so that they are inverses of one another.  In particular, the intermediate steps of the two procedures coincide after each column.  This holds true even if specific ``bumps" / ``reverse bumps" aren't direct reversals of one another when working within a specific column, as demonstrated by an inability to directly reverse specific steps in the examples of Figure \ref{fig: rectangular i=1 to i=0 example} or Figure \ref{fig: rectangular i=0 to i=1 example}.  Since $\phi_2 \circ \phi_1 (\tau) = \tau$ for all $\tau \in S_1(\lambda,\mu)$ and $\phi_1 \circ \phi_2 (T) = T$ for all $T \in S_0(\widetilde{\lambda},\mu)$, we may deduce that both maps are bijections and that $\vert S_1(\lambda,\mu) \vert = \vert S_0 (\widetilde{\lambda},\mu) \vert$.
\end{proof}

Theorem \ref{thm: enumeration i=1 rectangular} admits a generalization to non-rectangular tableaux that utilizes slight modifications of the procedures for $\phi_1$ and $\phi_2$.  As with Theorem 3.2 of \cite{BD}, this requires the introduction of additional terminology that describes the resulting shape change in the tableaux.

So consider the tableau shape $\lambda = (\lambda_1,\hdots,\lambda_m)$.  Define $d_i = \lambda_i - \lambda_{i+1}$ for $1 \leq i < m$ and $d_m = \lambda_m$, meaning that $d_i > 0$ if the $i^{th}$ row contains a ``lower-right corner" and $d_i = 0$ otherwise.  Then define $\widetilde{d}_i = \lambda_{i-1} - \lambda_i$ for $1 < i \leq m$ and $\widetilde{d}_1 = \infty$, so that $\widetilde{d}_i > 0$ if and only if an additional box may be added to the $i^{th}$ row without yielding an invalid tableau shape.  With $d_i$ and $\widetilde{d}_i$ defined for each row of the tableau shape $\lambda$, we have the following:

\begin{theorem}
\label{thm: enumeration i=1 general}
Consider the tableau shape $\lambda = (\lambda_1,\hdots,\lambda_m)$ with $m > 1$, and let $\mu$ be any content compatible with $\lambda$.  Then:
\begin{center}
$\displaystyle{\vert S_1((\lambda_1,\hdots,\lambda_m),\mu) \vert = \sum_E \vert S_0((\lambda_1 + \epsilon_1,\hdots, \lambda_m + \epsilon_m),\mu) \vert}$
\end{center}
\noindent Where the summation is over all tuples $E = (\epsilon_1,\hdots,\epsilon_m)$ such that $\epsilon_i= 0,\pm 1$ for all $i$, $\epsilon_i = -1$ for precisely one $i = i_1$ with $i_1 >1$ and $d_{i_1}>0$, and $\epsilon_i = 1$ for precisely one $i = i_2$ with $i_2 < i_1$ and $\widetilde{d}_{i_2}>0$.
\end{theorem}
\begin{proof}
Follows directly from the procedures of Theorem \ref{thm: enumeration i=1 rectangular}, as with the proof of Theorem 3.2 of \cite{BD}.  For an example of the shape change in this general bijection, which involves identifying all possible ending points of the $\phi_1$ procedure from the proof of Theorem \ref{thm: enumeration i=1 rectangular}, see Figure \ref{fig: general shape change comparison}.
\end{proof}

\begin{figure}[h!]
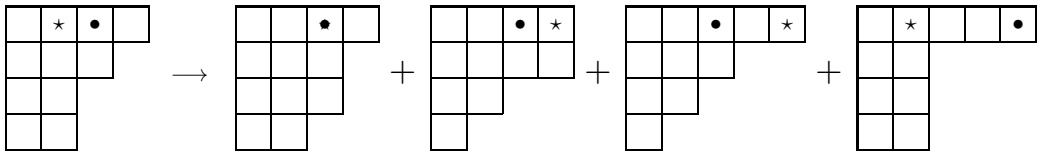

\centering
\scalebox{.8}{
\ydiagram{4,3,2,2}
*[\bullet]{0,2+1}
*[\star]{0,0,0,1+1}
\hspace{.1in}
\raisebox{-18.5pt}{$\longrightarrow$}
\hspace{.08in}
\ydiagram{4,3,3,2}
*[\bullet]{0,2+1}
*[\star]{0,0,2+1}
\hspace{.02in}
\raisebox{-18.5pt}{\scalebox{1.5}{$+$}}
\hspace{.0in}
\ydiagram{4,4,2,1}
*[\bullet]{0,2+1}
*[\star]{0,3+1}
\hspace{.02in}
\raisebox{-18.5pt}{\scalebox{1.5}{$+$}}
\hspace{.0in}
\ydiagram{5,3,2,1}
*[\bullet]{0,2+1}
*[\star]{4+1}
\hspace{.02in}
\raisebox{-18.5pt}{\scalebox{1.5}{$+$}}
\hspace{.0in}
\ydiagram{5,2,2,2}
*[\bullet]{4+1}
*[\star]{0,0,0,1+1}
}
\caption{Shape change in the bijection of Theorem \ref{thm: enumeration i=1 general} when $\lambda = (4,3,2,2)$}
\label{fig: general shape change comparison}
\end{figure}

\bibliographystyle{amsplain}
\bibliography{biblio}

\end{document}